\documentclass{amsart}

\usepackage{amssymb, amscd, latexsym, graphicx, psfrag}
\usepackage[all]{xy}

\newtheorem{dummy}{dummy}[section]
\newtheorem{lemma}[dummy]{Lemma}
\newtheorem{theorem}[dummy]{Theorem}

\newtheorem{corollary}[dummy]{Corollary}
\newtheorem{proposition}[dummy]{Proposition}
\theoremstyle{definition}
\newtheorem{definition}[dummy]{Definition}

\newtheorem{remark}[dummy]{Remark}

\newcommand{\sfR}{\mathsf{R}}

\newcommand{\bP}{\mathbb{P}}

\newcommand{\bR}{\mathbb{R}}
\newcommand{\bZ}{\mathbb{Z}}

\newcommand{\cD}{\mathcal{D}}
\newcommand{\cE}{\mathcal{E}}

\newcommand{\cH}{\mathcal{H}}

\newcommand{\cL}{\mathcal{L}}
\newcommand{\cP}{\mathcal{P}}
\newcommand{\cQ}{\mathcal{Q}}
\newcommand{\cO}{\mathcal{O}}

\newcommand{\Hom}{\mathrm{Hom}}
\newcommand{\dghom}{\mathit{hom}}			
\newcommand{\uhom}{\underline{\mathit{hom}}}

\newcommand{\LS}{ {\Lambda_\Sigma} }
\newcommand{\tLS}{\tilde{\Lambda}_\Sigma}

\newcommand{\Spec}{\mathrm{Spec}}
\newcommand{\Ext}{\mathrm{Ext}}

\newcommand{\Sh}{\mathit{Sh}}		
\newcommand{\Tr}{\mathit{Tr}}			
\newcommand{\Perf}{\cP\mathrm{erf}}

\renewcommand{\SS}{\mathit{SS}}		
\newcommand{\ltr}{\langle \Theta \rangle}

\newcommand{\ltrp}{\langle\Theta'\rangle}
\renewcommand{\lor}{\langle \oTheta \rangle}
\newcommand{\lorp}{\langle \oTheta' \rangle}

\newcommand{\CSi}{\mathbf{C}(\Sigma)}

\newcommand{\okappa}{\overline{\kappa}}

\newcommand{\forg}{\mathrm{forg}}
\newcommand{\oTheta}{\overline{\Theta}}

\newcommand{\ori}{\mathfrak{or}}
\newcommand{\bGamma}{\mathbf{\Gamma}}

\begin{document}
\title[The nonequivariant coherent-constructible correspondence]
{Remarks on the nonequivariant coherent-constructible correspondence for toric varieties}

\author{David Treumann}
\address{David Treumann, Department of Mathematics, Northwestern University,
2033 Sheridan Road, Evanston, IL  60208}
\email{treumann@math.northwestern.edu}

\maketitle

\section{Introduction}

The following is a theorem announced by Bondal \cite{bondal}:

\begin{theorem}
\label{thm:one}
Let $X$ be a complete $n$-dimensional toric variety.  There is a full embedding of dg categories
$$\okappa:\Perf(X) \hookrightarrow \Sh(\bR^n/\bZ^n)$$
where
\begin{itemize}
\item $\Perf(X)$ denotes the dg triangulated category of bounded complexes of vector bundles on $X$ (the so-called perfect complexes)
\item $\Sh(\bR^n/\bZ^n)$  denotes the bounded derived dg category of sheaves on $\bR^n/\bZ^n$.
\end{itemize}
\end{theorem}

Bondal considered the case where $X$ is smooth, so that $\Perf(X)$ coincides with the usual bounded derived category of coherent sheaves.  A natural problem is to identify, for a given toric variety $X$, those sheaves on $\bR^n/\bZ^n$ of the form $\okappa(\cE)$ for $\cE \in \Perf(X)$.

In \cite{fltz} we considered an equivariant version of theorem \ref{thm:one}: we gave a full embedding $\kappa$ of $\Perf_T(X)$ into $\Sh(\bR^n)$, where $\Perf_T(X)$ denotes the category of $T$-equivariant perfect complexes.  In this setting we were able to identify the image of $\kappa$ as the category of constructible sheaves with compact support whose ``singular support'' (a notion of \cite{KS}) is contained in a certain subset $\LS \subset T^* \bR^n$.  
The set $\LS$ is stable under $\bZ^n$-translation---denote its image in $T^*(\bR^n/\bZ^n)$ by $\tLS$.  It is natural to conjecture that the image of $\okappa$ consists of all sheaves on $T^*(M_\bR/M)$ with singular support in $\tLS$.

The purpose of this paper is to discuss this conjecture.  In section \ref{sec:two}, we do some bookkeeping to show that the techniques of \cite{fltz} suffice to construct $\okappa$, prove Bondal's theorem, compare $\okappa$ with its equivariant analogue $\kappa$, and verify a key functoriality result.  In section \ref{sec:three}, we construct a functor from $\Sh_c(\bR^n/\bZ^n;\tLS)$ to the bounded derived category of coherent sheaves on $X$.  This functor should be quasi-inverse to $\okappa$, and we discuss the obstacles to proving this.  In section \ref{sec:zonotopal}, we give a class of smooth toric varieties for which $\okappa$ can indeed be shown to be an equivalence.

\emph{Acknowledgements:}  Many of the results of this paper were known to Alexei Bondal in or before 2005 \cite{bondal}.  I am very grateful to Bohan Fang, Chiu-Chu Melissa Liu, and Eric Zaslow for a happy collaboration.  

\quad

\noindent
{\bf Notation and conventions.}  We use the notation and conventions of \cite{fltz}.

\section{$\Theta$-sheaves in a nonequivariant setting}
\label{sec:two}

In \cite{fltz}, we introduced objects $\Theta'(\sigma,\chi) \in \cQ_T(X)$ and $\Theta(\sigma,\chi) \in \Sh_c(M_\bR)$.  For $\sigma \subset N_\bR$ a cone and $\chi \subset M_\bR$ an integral coset of $\sigma^\perp$, we set
$$\begin{array}{ccc}
\Theta(\sigma,\chi) & = & ((\chi + \sigma^\vee)^\circ \hookrightarrow M_\bR)_! \ori[\dim(M_\bR)] \\
\Theta'(\sigma,\chi) & = & (X_\sigma \hookrightarrow X)_* \cO(\chi)
\end{array}
$$
The functor $\kappa$ is determined by the formula $\kappa(\Theta'(\sigma,\chi)) = \Theta(\sigma,\chi)$.  In this section we modify the definitions of these ``$\Theta$-sheaves'' for nonequivariant use, and verify that the main results of \cite[Section 3]{fltz} go through.  In fact all of the results in this section are essentially consequences of the equivariant results of loc. cit.

\subsection{The category $\CSi$ and the sheaves $\oTheta$ and $\oTheta'$}

The set of pairs $(\sigma,\chi)$ is denoted by $\bGamma(\Sigma,M)$.  It is endowed with a partial order by
$$(\sigma,\phi) \leq (\tau,\psi) \text{ if } \phi + \sigma^\vee \subset \psi + \tau^\vee$$
This poset carries an action of the group $M$, just by translating the second coordinate.  We can form a kind of quotient, which is a category we call $\CSi$:

\begin{definition}
Let $\Sigma$ be a rational polyhedral fan in $N_\bR$.
\begin{enumerate}

\item
Let $\CSi$ be the following category:
\begin{itemize}
\item The objects of $\CSi$ are the cones in $\Sigma$.
\item The hom sets are given by 
$$\Hom_{\CSi}(\sigma,\tau) = \bigg\{
\begin{array}{ll}
(\tau^\vee \cap M)/(\tau^\perp \cap M) & \text{if $\sigma^\vee \subset \tau^\vee$}\\
\varnothing & \text{otherwise}
\end{array}
$$
\item The composition law is induced by the addition maps
$$(\sigma^\vee \cap M) \times (\tau^\vee \cap M) \to (\upsilon^\vee \cap M)$$
whenever $\sigma^\vee \subset \tau^\vee \subset \upsilon^\vee$. 
\end{itemize}
\item
Let $(\CSi)_\sfR$ denote the $\sfR$-linear category spanned by $\CSi$. 
\end{enumerate}

\end{definition}

Note that we have a functor from the poset $\bGamma(\Sigma,M)$ to the category $\CSi$ defined in the following way:
\begin{itemize}
\item The object $(\sigma,\chi)$ maps to the object $\sigma$
\item Whenever $(\sigma,\phi) \leq (\tau,\psi)$ in the partial order, the unique map from $(\sigma,\phi)$ to $(\tau,\psi)$ gets sent to
the map $\overline{\phi} - \psi \in (\tau^\vee \cap M)/(\tau^\perp \cap M)$ where $\overline{\phi}$ is the $\tau^\perp$-coset spanned by the $\sigma^\perp$-coset $\phi$.
\end{itemize}

Let $p$ denote the map $M_\bR \to M_\bR/M$, and let $\forg$ denote the forgetful functor $\cQ_T(X) \to \cQ(X)$.  For each $\sigma \in \Sigma$, fix a coset $\chi$ of $\sigma^\perp$ and set
$$\begin{array}{ccc}
\oTheta(\sigma) & = & p_! \Theta(\sigma,\chi) \\
\oTheta'(\sigma) & = & \forg (\Theta(\sigma,\chi))
\end{array}
$$
These objects are independent of $\chi$, up to a unique quasi-isomorphism.

\begin{remark}
The sheaves $\oTheta(\sigma) \in \Sh(M_\bR/M)$ have infinite-dimensional fibers, but they are still constructible in the weak sense that they are locally constant on some stratification of $M_\bR/M$.  Let us call sheaves with this property \emph{quasiconstructible}, and denote the full triangulated dg subcategory of $\Sh(M_\bR/M)$ consisting of quasiconstructible sheaves by $\Sh_{qc}(M_\bR/M)$.
\end{remark}

The assignments $\sigma \mapsto \oTheta(\sigma)$ and $\sigma \mapsto \oTheta'(\sigma)$ may be extended to fully faithful dg functors $(\CSi)_{\sfR} \to \Sh_{qc}(M_\bR/M)$ and $(\CSi)_{\sfR} \to \cQ(X)$, as follows.  If $\sigma$ and $\tau$ are cones with $\sigma^\vee \subset \tau^\vee$ and $\chi$ is an element of $\tau^\vee \cap M$, let $f_{\sigma,\tau,\chi}$ denote the map $\oTheta(\sigma) \to \oTheta(\tau)$ obtained by pushing forward the inclusion $\Theta(\sigma,\chi) \to \Theta(\tau,0)$ under $p_!$.  Similarly, let $f'_{\sigma,\tau,\chi}:\oTheta'(\sigma) \to \oTheta'(\tau)$ be the image under $\mathrm{forg}$ of the inclusion $\Theta'(\sigma,\chi) \to \Theta'(\tau,0)$.  

\begin{proposition}
\label{prop:33}
Let $\sigma$ and $\tau$ be cones in $\Sigma$.  Then 
\begin{enumerate}
\item
$$\Ext^i(\oTheta(\sigma),\oTheta(\tau)) \cong \bigg\{
\begin{array}{ll}
\sfR[\tau^\vee \cap M] & \text{if $i = 0$ and $\sigma^\vee \subset \tau^\vee$} \\
0 & \text{otherwise}
\end{array}
$$
\item
$$\Ext^i(\oTheta'(\sigma),\oTheta'(\tau)) \cong \bigg\{
\begin{array}{ll}
\sfR[\tau^\vee \cap M] & \text{if $i = 0$ and $\sigma^\vee \subset \tau^\vee$} \\
0 & \text{otherwise}
\end{array}
$$
\end{enumerate}
\end{proposition}

When the $\Ext^0$ group is nonvanishing, the identifications with $\sfR[\tau^\vee \cap M]$ are given by $f_{\sigma,\tau,\chi} \mapsto \chi$ and $f'_{\sigma,\tau,\chi} \mapsto \chi$.

\begin{proof}
We have
$$\dghom(\oTheta(\sigma),\oTheta(\tau)) := \dghom(p_!(\Theta(\sigma,\chi)),\oTheta(\tau)) \cong \dghom(\Theta(\sigma,\chi),p^! \oTheta(\tau))$$
The maps $p^! \oTheta(\tau) \to \Theta(\tau,\xi)$, for $\xi \in M$, define a quasi-isomorphism $p^!(\oTheta(\tau)) \cong \bigoplus_{\xi \in M} \Theta(\tau,\xi)$.  Thus we have
$$\dghom(\oTheta(\sigma),\oTheta(\tau)) \cong \bigoplus_{\xi \in M} \dghom(\Theta(\sigma,\chi),\Theta(\tau,\xi))$$
Part (1) now follows from proposition 3.3 of \cite{fltz}; part (2) may be proved similarly.
\end{proof}

\begin{theorem}
\label{thm:square}
Let $X$ be a toric variety with fan $\Sigma$.  Let $\ltr \subset \Sh_c(M_\bR)$ and $\ltrp \subset \cQ_T(X)$ denote the full triangulated dg subcategories generated respectively by the objects $\Theta$ and $\Theta'$, and let $\kappa:\ltrp \stackrel{\sim}{\to} \ltr$ be the quasiequivalence of \cite[Theorem 3.4]{fltz}.  Let $\lor \subset \Sh_{qc}(M_\bR/M)$ denote the full triangulated dg subcategory generated by the objects $\oTheta(\sigma)$, and let $\lorp \subset \cQ(X)$ denote the full triangulated subcategory generated by the objects $\Theta'(\sigma,\chi)$. There exists a quasi-equivalence $\okappa:\ltrp \stackrel{\sim}{\to} \ltr$ that makes the following diagram commute up to isomorphism:
$$\xymatrix{
\ltrp \ar[r]^\kappa \ar[d]_{\mathrm{forg}} & \ltr \ar[d]^{p_!} \\
\lorp \ar[r]_{\overline{\kappa}} & \lor
}$$
\end{theorem}

\begin{proof}
By proposition \ref{prop:33}, the functors $\CSi_\sfR \to \Sh_{qc}(M_\bR/M):\sigma \mapsto \oTheta(\sigma)$ and $\CSi_\sfR \to \cQ(X):\sigma \mapsto \oTheta'(\sigma)$ are fully faithful.  Therefore there is a diagram of quasi-equivalences
$$\lorp \stackrel{\sim}{\leftarrow} \mathit{Tr}(\CSi_\sfR) \stackrel{\sim}{\rightarrow} \lor$$
where $\Tr$ denotes the triangulated envelope of $\CSi_\sfR$.  We let $\okappa$ be the composition.  The diagram in the theorem commutes because the following diagram of quasi-equivalences, in which the top row defines $\kappa$ and the bottom row defines $\okappa$, commutes by definition:
$$\xymatrix{
\ltrp \ar[d]_{\mathrm{forg}}  & \mathit{Tr}(\bGamma(\Sigma,M)_\sfR) \ar[l] \ar[r] \ar[d] &  \ltr \ar[d]^{p_!} \\
\lorp  & \mathit{Tr}(\CSi_\sfR) \ar[l] \ar[r] & \lor
}$$

\end{proof}

\subsection{Functoriality}
\label{sec:functoriality}

Recall \cite[3.6]{fltz} that to a fan-preserving map $f:N_1 \to N_2$ between $(N_1,\Sigma_1)$ and $(N_2,\Sigma_2)$, we associate a morphism of toric varieties $u = u_f:X_1 \to X_2$.  We also have the map $v = v_f:M_{2,\bR} \to M_{1,\bR}$, which descends to a map $M_{2,\bR}/M_2 \to M_{1,\bR}/M_1$.  We abuse notation and denote this map by $v$ as well.

\begin{theorem}
Let $f$ be a fan-preserving map from $\Sigma_1 \subset N_{1,\bR}$ to $\Sigma_2 \subset N_{2,\bR}$.  Suppose that $f$ furthermore satisfies the following conditions:
\begin{enumerate}
\item The inverse image of any cone $\sigma_2 \subset \Sigma_2$ is a union of cones in $\Sigma_1$.
\item $f$ is injective.
\end{enumerate}
Let $u$ and $v$ be as above.  Then
\begin{enumerate}
\item 
The pullback $u^*:\cQ(X_2) \to \cQ(X_1)$ takes $\lorp_2$ to $\lorp_1$.
\item  The pushforward $v_!:\Sh(M_{2,\bR}/M_2) \to \Sh(M_{1,\bR}/M_1)$ carries $\lor_2$ to $\lor_1$.

\item  The following square of functors commutes up to natural isomorphism:
$$
\xymatrix{
\lorp_2 \ar[r]^{\okappa_2} \ar[d]_{u^*} & \lor_2 \ar[d]^{v_!} \\
\lorp_1 \ar[r]_{\okappa_1} & \lor_1
}
$$
\end{enumerate}
\end{theorem}

\begin{proof}
We will construct a natural transformation $\overline{\iota}:v_! \circ \okappa_2 \to \okappa_1 \circ u^*$, and conclude that $\overline{\iota}$ is a quasi-isomorphism by appealing to the fact that it is a quasi-isomorphism in the equivariant setting.  It suffices to define $\overline{\iota}$ on the generators $\oTheta'(\sigma)$.  We have
natural quasi-isomorphisms
$$
\begin{array}{c}
v_! \circ \okappa_2(\oTheta'(\sigma)) \cong v_! \oTheta(\sigma)[d] = v_! p_{2!} \Theta(\sigma,0)[d] \cong p_{1!} v_! \Theta(\sigma,0)[d]\\
\okappa_1 \circ u^*(\oTheta'(\sigma)) \cong \okappa_1 \circ \mathrm{forg}(u^* \Theta'(\sigma,0)) \cong
p_{1!} \kappa_1 u^* \Theta'(\sigma,0) 
\end{array}$$
We may therefore define $\overline{\iota}$ to be the image under $p_{1!}$ of the equivariant analog $\iota:v_! \circ \kappa_2 \to \kappa_1 \circ u^*$ defined in the proof of \cite[Theorem 3.8]{fltz}.  Since $\iota$ is a quasi-isomorphism, so it $\overline{\iota}$.  This completes the proof.
\end{proof}

As in the equivariant case \cite[Example 3.12]{fltz}, applying the theorem to the diagonal map $\Delta:X \to X \times X$ shows that $\okappa$ induces an equivalence of monoidal categories $h\lor \stackrel{\sim}{\to} h\lorp$, where the tensor structure on $\lor$ is the usual tensor product of quasicoherent sheaves, and the tensor structure on $\lorp$ is given by convolution.

\subsection{Perfect complexes and $\tLS$-sheaves}
\label{subsec:perfect}

By studying the \v Cech complex we can see that the category $\lorp$ contains $\Perf(X)$.

\begin{proposition}
Let $X$ be a toric variety with fan $\Sigma$.  If $\cE$ is a vector bundle on $X$, then $\cE$ is quasi-isomorphic to a bounded complex of quasicoherent sheaves
of the form $\bigoplus_{i=1}^n \oTheta'(\sigma_i)$.
\end{proposition}

\begin{proof}
The proposition is clear for trivial vector bundles.  Moreover, if the proposition is true for $\cE \oplus \cO^{\oplus n}$, then it is true for $\cE$, since the latter is the cone on a map
$$\cO^{\oplus n} \to \cE \oplus \cO^{\oplus n}$$
Now recall the fact that on an affine toric variety, every vector bundle is stably trivial.  (In fact, it is known that every vector bundle on an affine toric variety is actually trivial, \cite{gub}).  Thus we may find $n$ such that $\cE \oplus \cO^n$ admits a trivialization over every chart.  It follows that in the \v Cech complex for $\cE \oplus \cO^n$ we may write each term as $\bigoplus_i \oTheta(\sigma_i)$.   This completes the proof.
\end{proof}

In contrast, we do not know if the category $\Sh_c(M_\bR/M;\tLS)$ is contained in $\lor$.  However we can show

\begin{proposition}
The image of $\Perf$ under $\overline{\kappa}$ is contained in $\Sh_c(M_\bR/M;\LS)$
\end{proposition}

\begin{proof}
The strategy is similar to that of the proof of \cite[Theorem 7.1]{fltz}.  We use the functoriality results of section \ref{sec:functoriality} to reduce to the case where $X$ is smooth and projective.  In that case, $\Perf(X)$ is generated by line bundles.  Every line bundle can be given an equivariant structure, so that from the square of theorem \ref{thm:square}, we have $\overline{\kappa}(\cL) = p_! \kappa(\cL')$ where $\cL'$ is an equivariant line bundle.  We showed in \cite{fltz} that $F = \kappa(\cL')$ has compact support and that $\SS(F)$ belongs to $\LS$.   It follows that $\SS(p_! F) \subset p(\LS) = \tLS$.
\end{proof}

\section{A candidate inverse to $\okappa$}
\label{sec:three}

The results section \ref{subsec:perfect} give us a commutative square:
$$
\xymatrix{
\Perf_T(X) \ar[r]^\kappa \ar[d]_{\forg} & \Sh_{cc}(M_\bR;\LS) \ar[d]^{p_!} \\
\Perf(X) \ar[r]_\okappa & \Sh_c(M_\bR/M;\tLS)
}
$$
We have seen that the rows in this square are fully faithful.  In \cite{fltz} we showed that the top row is a quasi-equivalence of dg categories.  To show that a sheaf $F$ on $M_\bR$ is in the image of $\kappa$, we inducted on, roughly speaking, the number of ``critical values'' that $F$ has with respect to linear functions $M_\bR \to \bR$.  We believe that the bottom row is a quasi-equivalence as well, but there is no simple generalization of the same induction argument: real-valued functions on a torus are more complicated than linear functions on $M_\bR$.

In this section, we see how far some easier arguments can go toward establishing the essential surjectivity of $\okappa$.  We construct a functor $\rho:\Sh(M_\bR/M) \to \cQ(X)$ which when restricted to $\Sh_c(M_\bR/M;\tLS)$ should be a kind of inverse to $\okappa$.  More precisely we show that $\rho \circ \okappa$ takes a perfect complex $\cE$ to the dual perfect complex $\cE^\vee = \uhom(\cE,\cO)$, and we use this to make some statements (corollary \ref{cor:conjecture}) equivalent to the essential surjectivity of $\okappa$.

Before defining $\rho$, we discuss duality.  Write $\cE \mapsto \cE^\vee$ for the natural duality functor (i.e. $\cE^\vee = \uhom(\cE,\cO)$) on $\Perf(X)$ and write $\cD$ for the Verdier dualizing functor on $\Sh_c(M_\bR/M)$.  Note that $\cD$ composed with the map $-1:M_\bR/M \to M_\bR/M$ preserves the subcategory $\Sh_c(M_\bR/M;\LS)$.  The proof of \cite[Theorem 7.4]{fltz} establishes the following

\begin{theorem}
There is a natural isomorphism
$$\okappa(\cE^\vee) \cong -\cD(\okappa(\cE))$$
\end{theorem}

We will also need of variant of Gordon's lemma.  Let $\sigma^\vee \subset M_\bR$ be a rational polyhedral cone.  For each $x \in M_\bR$, consider the coset $x + M$ of $M \subset M_\bR$, and construct the $\sfR[\sigma^\vee \cap M]$-module whose underlying vector space has a basis given by the points of $(x + M) \cap (\sigma^\vee)^\circ$, and where the $\sfR[\sigma^\vee \cap M]$ action is given by
$$\chi \cdot [x + \xi] = \bigg\{
\begin{array}{ll} [x + \xi + \chi] & \text{if }x + \xi + \chi \in (\sigma^\vee)^\circ \\
0 & \text{otherwise}
\end{array}$$

\begin{lemma}
\label{lem:fingen}
The $\sfR[\sigma^\vee \cap M]$-module just described is finitely generated.
\end{lemma}

\begin{proof}
The proof is similar to the proof of Gordon's lemma given in \cite[p. 12]{Fu}.
  Since $\sfR[\sigma^\vee \cap M]$ is a tensor product of $\sfR[\sigma^\perp \cap M]$ and $\sfR[\sigma^\vee/\sigma^\perp \cap M/(\sigma^\perp \cap M)]$, we may reduce to the case where $\sigma^\vee$ is strongly convex (i.e. to the case where $\sigma \subset N_\bR$ is full-dimensional).  Let $u_1,\ldots,u_s$ be vectors in $M$ whose $\bR_{\geq 0}$-span is all of $\sigma^\vee$, and let $K = \{\sum t_i u_i \mid 0 \leq t_i \leq 1\}$.  Since $K$ is compact and $(x+ M) \cap (\sigma^\vee)^\circ$ is discrete, the intersection $K \cap (x + M) \cap (\sigma^\vee)^\circ$ is finite.  We claim that this finite set generates the module.  To see this, note that we can write any $u \in (x+ M) \cap (\sigma^\vee)^\circ$ as $u = \sum m_i u_i + \sum t_i u_i$ with $m_i \in \bZ$ and $t_i \in [0,1)$.  Putting $\chi = \sum m_i u_i$, we have $[u] = \chi \cdot [\sum t_i u_i]$.
\end{proof}

For every vector bundle $\cE$ on $X$, the vector space $\Hom(\cE,\oTheta'(\sigma))$ is a module over the ring $\text{End}(\oTheta'(\sigma),\oTheta'(\sigma)) \cong \sfR[M \cap \sigma^\vee]$.  This module is naturally identified with the $\sfR[M \cap \sigma^\vee]$-module associated to the vector bundle $\cE^\vee\vert_{X_\sigma}$ on the affine variety $X_\sigma$.  Then $\cE^\vee$ is quasi-isomorphic to \v Cech complex
$$\bigoplus_{i_0} j_{C_{i_0}*}\Hom(\cE,\oTheta'(C_{i_0})) \to \bigoplus_{i_0 < i_1} j_{C_{i_0} \cap C_{i_1}*} \Hom(\cE,\oTheta'(C_{i_0} \cap C_{i_1}) \to \cdots$$
where $C_1,\ldots,C_v$ is some ordering of the maximal cones in $\Sigma$.

This motivates the following definition:

\begin{definition}
For $F \in \Sh(M_\bR/M)$, let $\rho(F)$ be the total complex of the bicomplex
$$\bigoplus_{i_0} j_{C_{i_0}*}\dghom(F,\oTheta(C_{i_0})) \to \bigoplus_{i_0 < i_1} j_{C_{i_0} \cap C_{i_1}*} \dghom(F,\oTheta(C_{i_0} \cap C_{i_1})) \to \cdots$$
where each $\dghom(F,\oTheta(\sigma))$ is regarded as a complex of quasicoherent sheaf on $X_\sigma = \Spec(\text{End}(\oTheta(\sigma)))$, pushed forward along the inclusion $X_\sigma \hookrightarrow X$.
\end{definition}

\begin{remark}
\label{rem:conv1}
It might be more natural to consider the covariant functor which sends $F$ to the quasicoherent sheaf glued together from the $\text{End}(\oTheta(\sigma))$-modules given by $\hom(\delta,F \star \oTheta(\sigma))$.  When $F$ is constructible, this construction coincides with $\rho(-\cD(F))$.  The construction we have given avoids the convolution product $\star$.
\end{remark}

\begin{theorem}
\begin{enumerate}
\item
We have $\rho(\kappa(\cE)) \cong \cE^\vee$ for all perfect complexes $\cE$.

\item
If $F \in \Sh_c(M_\bR/M)$, then $\rho(F)$ is quasi-isomorphic to a bounded complex of coherent sheaves.
\end{enumerate}
\end{theorem}

\begin{proof}
The first assertion follows from the existence of a \v Cech resolution and from proposition \ref{prop:33}.  To prove the second assertion, we can make two reductions:
\begin{itemize}
\item It suffices to show that $\hom(F,\oTheta(\sigma))$ is a perfect complex of $\mathrm{End}(\oTheta(\sigma))$-modules for each $\sigma$.
\item We may furthermore assume that $F$ is a costandard sheaf on the star of a simplex in some triangulation of $M_\bR/M$, and we may take this triangulation to be as fine as we wish.
\end{itemize}
When $F$ is such a costandard sheaf, the space $\hom(F,\oTheta(\sigma))$ is isomorphic to the stalk of $\oTheta(\sigma)$ at some point $x \in M_\bR/M$.  We therefore have to show that $\oTheta(\sigma)_x$ is a finitely-generated $\text{End}(\oTheta(\sigma))$-module.  By the proper base change theorem, we have $\oTheta(\sigma)_x \cong \bigoplus_{y \in M_\bR \mid p(y) = x} \Theta(\sigma,0)_y = \sfR[(\sigma^\vee)^\circ \cap p^{-1}(x)]$.  This is the module considered in lemma \ref{lem:fingen}, which completes the proof.
\end{proof}

\begin{corollary}
\label{cor:conjecture}
For a smooth, complete toric variety $X$ with associated fan $\Sigma$, the following are equivalent:

\begin{enumerate}
\item $\okappa:\Perf(X) \to \Sh_c(M_\bR/M;\LS)$ is a quasi-equivalence.
\item If $F \in \Sh_c(M_\bR/M;\tLS)$ and $\rho(F) = 0$, then $F = 0$.
\item If $F \in \Sh_c(M_\bR/M;\tLS)$ and $\dghom(F,\oTheta(\sigma)) = 0$ for all $\sigma$, then $F = 0$.
\end{enumerate}
\end{corollary}

\begin{proof}
If $\okappa$ is a quasi-equivalence, then $F = \okappa(\cE)$ for some perfect complex $\cE$, and we have $\rho(F) = \cE^\vee$.  Thus $\rho(F) = 0$ implies that $F = \kappa(\cE) = 0$.  This shows that (1) implies (2).  Conversely, if (2) holds then $\okappa$ is essentially surjective, and since we have already shown it to be fully faithful (1) holds as well.  To show that (2) is equivalent to (3), note that $\rho(F)$ vanishes if and only if  $\rho(F)\vert_{X_\sigma}$ vanishes for all $\sigma$.  But by definition $\rho(F)\vert_{X_\sigma}$ is the quasicoherent sheaf associated to the complex of $\text{End}(\oTheta(\sigma))$-modules $\dghom(F,\oTheta(\sigma))$.

\end{proof}

\begin{remark}
If the variant of $\rho$ described in remark \ref{rem:conv1} can be shown to respect the monoidal structure (i.e. to intertwine the convolution and tensor products) then the smoothness hypothesis of corollary \ref{cor:conjecture} could be removed.  Indeed, the objects of $\Sh_c(M_\bR/M;\tLS)$ are strongly dualizable (see \cite[Proposition 7.3]{fltz}), so in that case their images under $\rho$ would also be strongly dualizable and therefore perfect.  
\end{remark}

\begin{remark}
The statement (3) of the corollary is similar to, but not obviously stronger or weaker than, conjecture 5.8 of \cite{fltz}.
\end{remark}

\section{Toric varieties from unimodular zonotopes}
\label{sec:zonotopal}

In this section we describe a class of toric varieties for which the conditions of corollary \ref{cor:conjecture} hold.  Let $Y$ be a manifold and let $\Lambda \subset T^* Y$ be a conical Lagrangian.  One of the things that makes working with $\Sh_c(Y;\Lambda)$ difficult is the lack of a good collection of generators for the category.  However there are certain kinds of conical Lagrangians for which some generators can be found.  Here we classify those $\Sigma$ for which $\Lambda_\Sigma$ is the conormal variety to a Whitney stratification of $M_\bR/M$, in which case the costandard sheaves on strata generate.  

\begin{remark}
Already in \cite{bondal}, Bondal found a more general class of fans $\Sigma$ for which $\Perf(X)$ is equivalent  to the category of sheaves that are constructible with respect to a non-Whitney stratification of $M_\bR/M$.  It is possible to find generators of $\Sh_c(M_\bR/M;\tLS)$ for these examples as well (in fact they are described in \cite{bondal}), but here we consider only the Whitney case.
\end{remark}

\begin{definition}
A projective toric variety $X$ is called \emph{zonotopal} if either of the following equivalent conditions hold:
\begin{enumerate}
\item Let $\Delta \subset M_\bR$ denote the moment polytope of $X$ with respect to some polarization.  Every face of $\Delta$ is symmetric about its barycenter.
\item Let $\Sigma \subset N_\bR$ denote the fan associated to $X$.  Then $\Sigma$ is exactly the fan of chambers associated to a hyperplane arrangement in $N_\bR$.
\end{enumerate}
(See \cite{zono} for a proof that these conditions are indeed equivalent.)  A zonotopal toric variety of dimension $n$ is called \emph{unimodular} if any $n$ linearly independent vectors chosen from the generators of the rays of $\Sigma$ form a $\bZ$-basis for $M$.
\end{definition}

For example, both $\bP^1 \times \bP^1$, and the toric variety obtained from $\bP^2$ by blowing up the three torus-fixed points are zonotopal and unimodular.  The toric variety obtained from $\bP^1 \times \bP^1$ by blowing up the four torus-fixed points is zonotopal, but not unimodular.  Unimodular toric varieties are always smooth.

Let $\cH$ be the arrangement of affine hyperplanes in $M_\bR$ that pass through a lattice point and that are orthogonal to one of the rays in $\Sigma$.  This arrangement is periodic---i.e. invariant under translation by lattice points---and induces a periodic polyhedral cell decomposition of $M_\bR$:

\begin{definition}
Let $\cH$ be a periodic hyperplane arrangement in $M_\bR$.  Define a filtration $\cH^{(i)}$ of $M_\bR$ as follows:
\begin{enumerate}
\item $\cH^{(0)} = M_\bR$
\item $\xi \in \cH^{(i)}$ if and only if $\xi$ belongs to the intersection of at least $i$ distinct hyperplanes in the arrangement, and these hyperplanes can moreover be chosen to have normal crossings.
\end{enumerate}
This filtration is a Whitney stratification of $M_\bR$; we will denote this stratification by $\cH$.
\end{definition}

\begin{theorem}
Let $X$ be a smooth projective zonotopal toric variety, with fan $\Sigma \subset N_\bR$. Let $\cH$ be the corresponding stratification of $M_\bR$.  If $X$ is unimodular, then the characteristic variety $\Lambda_\cH \subset T^*M_\bR = M_\bR \times N_\bR$ of the stratification $\cH$ is exactly $\LS$:
$$\Lambda_\cH = \LS := \bigcup_{\chi \in M} \bigcup_{\sigma \in \Sigma} (\chi + \sigma^\perp) \times (-\sigma)$$
\end{theorem}

\begin{proof}
From the results of \cite{fltz}, it is clear that $\LS \subset \Lambda_{\cH}$ for any fan $\Sigma$, since every sheaf in $\Sh_{cc}(M_\bR;\LS)$ is constructible with respect to the stratification $\cH$.  We will show that $\Lambda_\cH \subset \LS$.

For $\xi \in M_\bR$, set $\Lambda_{\cH,\xi} = \Lambda_{\cH} \cap \{\xi\} \times N_\bR$ and $\Lambda_{\Sigma,\xi} = \LS \cap \{\xi\} \times N_\bR$.  To prove the theorem it suffices to show that $\Lambda_{\cH,\xi} = \Lambda_{\Sigma,\xi}$ for all $\xi \in M_\bR$.  Suppose that $\xi \in \cH^{(i)} - \cH^{(i+1)}$, so that there exist hyperplanes $H_1,\ldots,H_i$ such that locally near $\xi$ the stratum containing $\xi$ coincides with $\xi + H_1 \cap \cdots \cap H_i$.  (The definition also requires that the affine hyperplanes $\xi + H_j$ coincide with affine hyperplanes of the form $\chi_j + H_j$, where $\chi_j$ is a lattice point.)  Then, by definition, $x$ belongs to $\Lambda_{\cH,\xi}$ if and only if $\langle \lambda,x\rangle = 0$ for all $\lambda \in H_1 \cap \cdots \cap H_i$.  On the other hand, $x$ belongs to $\Lambda_{\Sigma,\xi}$ if and only if the following holds: there exists a cone $\sigma$ such that $\langle \xi,-\rangle$ takes integer values on $\sigma \cap M$ and such that $-x \in \sigma$.

Each of the hyperplanes $H_i$ is perpendicular to one of the lines $\ell_i$ in $\Sigma$.  (This line is a union of a ray in $\Sigma$ and its antipode, which is also a ray.)  Suppose that $x$ belongs to $\Lambda_{\cH,\xi}$, so that $x$ is perpendicular to the intersection $H_1 \cap \cdots \cap H_i$.  Then $x$ belongs to the sum $\ell_1 + \cdots + \ell_i$, which is a linear space.  By the hypotheses on the fan, this linear space necessarily contains a cone $\sigma$ with $x \in -\sigma$.  To show that $\Lambda_{\cH,\xi} \subset \Lambda_{\Sigma,\xi}$, it remains to show that $\xi$ takes integer values on $\sigma \cap M$, or equivalently on $(\ell_1 + \cdots + \ell_i) \cap M$.  Let $x_j$ be a generator for $\ell_j \cap M \cong \bZ$.  Since $X$ is unimodular, $\{x_1,\ldots,x_i\}$ forms a $\bZ$-basis for $(\ell_1 + \cdots + \ell_i) \cap M$, and $\langle \xi, x_j\rangle = \langle \chi_j,x_j\rangle \in \bZ$ for each $j$.  This completes the proof.

\end{proof}

\begin{corollary}
If $X$ is a smooth projective zonotopal toric variety, then
$$\okappa:\Perf(X) \to \Sh_c(M_\bR/M,\tLS)$$
is a quasi-equivalence of dg categories.
\end{corollary}

\begin{proof}
We have already shown that $\okappa$ is fully faithful.  We will show that it is essentially surjective by showing that is generated by certain costandard sheaves which are in the image of $\okappa$.

Since $\LS$ is the characteristic variety of the Whitney stratification $\cH$, we may identify $\Sh_{cc}(M_\bR;\LS)$ with $\Sh_{cc}(M_\bR;\cH)$---i.e. with the dg category of compactly-supported sheaves that are constructible with respect to $\cH$.  Similarly, we may identify $\Sh_c(M_\bR/M;\tLS)$ with the dg category of sheaves that are constructible with respect to $\cH$.   Since each stratum of $\cH$ and $\cH/M$ is contractible, it follows that $\Sh_{cc}(M_\bR;\LS)$ and $\Sh_c(M_\bR/M;\tLS)$ are generated by the costandard sheaves $i_! \sfR$, where $i$ denotes the inclusion of a stratum.  Since every stratum of $M_\bR/M$ is covered by a stratum of $M_\bR$, to show that costandard sheaves are in the image of $\okappa$ it suffices to show that costandard sheaves are in the image of $\kappa$.  But this is a consequence of the main theorem of \cite{fltz}.
\end{proof}

\end{document}